\documentclass[reqno,oneside]{amsart}
\usepackage{graphicx} 

\usepackage[utf8]{inputenc}
\usepackage[super]{nth}
\usepackage{xcolor}
\usepackage{comment,fancyhdr,tikz,amsmath,amssymb,amsthm,mathtools,centernot,caption,float,graphicx,makecell,array,float,changepage,verbatim,textcomp,parallel,geometry,ragged2e,xcolor,lastpage,bbm,mathrsfs,pgfkeys,mathabx}
\newcommand\join\vee
\newcommand\meet\wedge

\usepackage[
    cal=boondoxo,
    scr=boondox,
    frak=euler,
    bb=ncmbbr]{mathalpha}
\usepackage{bm}

    \usepackage[shortlabels]{enumitem}

\usepackage[export]{adjustbox}
\usepackage[
    colorlinks=false,
    citebordercolor=green,
    linkbordercolor=red
    ]{hyperref}

    \usepackage[capitalise]{cleveref}
    \crefformat{equation}{(#2#1#3)}
    \crefformat{enumi}{(#2#1#3)}
    \crefformat{section}{\S#2#1#3}
    \crefformat{subsection}{\S#2#1#3}
    \crefformat{subsubsection}{\S#2#1#3}
    \numberwithin{equation}{section}
    \newcommand{\crefdefpart}[2]{%
        \hyperref[#2]{\namecref{#1}~\labelcref*{#1}~\ref*{#2}}}
    
    \let\etoolboxforlistloop\forlistloop 
    \usepackage{autonum}
    \let\forlistloop\etoolboxforlistloop 
    \makeatletter
    \newcommand{\blx@noerroretextools}{}
    \makeatother
    
    \usepackage[
        backend=biber,
        style=alphabetic,
        maxbibnames=50,
        maxalphanames=50,
        maxcitenames=50
        ]{biblatex}
    \newcommand{\Z}{\mathbb{Z}}
    
    \newcommand{\R}{\mathbb{R}}
    \newcommand{\C}{\mathbb{C}}
    
    \def\sph^#1{\mathbb S^{#1}}
    \newcommand\p[1]{\left(#1\right)}
    \newcommand\abs[1]{\left|#1\right|}

    \newcommand\jp[1]{\left\langle#1\right\rangle}
    
    \newcommand\eps\varepsilon
    \newcommand\pji\varphi
    \newcommand\sig\varsigma

    \newcommand\beq{\begin{equation}}
    \newcommand\eeq{\end{equation}}
    \newcommand\inv{^{-1}{}}

    \def\XXint#1#2#3{{\setbox0=\hbox{$#1{#2#3}{\int}$ }
    \vcenter{\hbox{$#2#3$ }}\kern-.6\wd0}}

    \newtheoremstyle{thmst}
        {5pt}
        {5pt}
        {}
        {}
        {\bfseries}
        {.}
        {.5em}
        {}%
    \newtheoremstyle{rmkst}
        {5pt}
        {5pt}
        {}
        {}
        {\bfseries}
        {.}
        {.5em}
        {}%
    \newtheoremstyle{lmst}
        {5pt}
        {5pt}
        {}
        {}
        {\bfseries}
        {.}
        {5pt plus 1pt minus 1pt}
        {}%
    \newtheoremstyle{prpst}
        {5pt}
        {5pt}
        {}
        {}
        {\bfseries}
        {.}
        {.5em}
        {}%
    \newtheoremstyle{defst}
        {5pt}
        {5pt}
        {}
        {}
        {\bfseries}
        {.}
        {.5em}
        {}%
    \newtheoremstyle{wlst}
        {5pt}
        {5pt}
        {}
        {}
        {\bfseries}
        {.}
        {.5em}
        {}%
    \newtheoremstyle{asmpst}
        {5pt}
        {5pt}
        {}
        {}
        {\bfseries}
        {.}
        {.5em}
        {}%
        
    \theoremstyle{thmst}
        \newtheorem{theorem}{Theorem}[section]

        \crefname{conjecture}{Conjecture}{Conjectures}
    \theoremstyle{rmkst}
        \newtheorem{remark}[theorem]{Remark}
    \theoremstyle{defst}
        \newtheorem{definition}[theorem]{Definition}
    \theoremstyle{defst}
        
    \theoremstyle{lmst}
        \newtheorem{lemma}[theorem]{Lemma}
        
    \theoremstyle{prpst}
        
    \theoremstyle{wlst}
        
    \theoremstyle{asmpst}
        
    \theoremstyle{thmst}

        \newlist{defenum}{enumerate}{1} 
        \setlist[defenum]{label=(\roman*),ref=\thedefinition\,(\roman*)}
        \crefname{defenumi}{Definition}{Definitions}
        
        \newlist{lemenum}{enumerate}{1} 
        \setlist[lemenum]{label=(\roman*),ref=\thelemma\,(\roman*)}
        \crefname{lemenumi}{Lemma}{Lemmas}
        
        \newlist{corenum}{enumerate}{1} 
        \setlist[corenum]{label=(\roman*),ref=\thecorollary\,(\roman*)}
        \crefname{corenumi}{Corollary}{Corollaries}
        
        \newlist{thmenum}{enumerate}{1} 
        \setlist[thmenum]{label=(\roman*),ref=\thetheorem\,(\roman*)}
        \crefname{thmenumi}{Theorem}{Theorems}
        
        \newlist{rmkenum}{enumerate}{1} 
        \setlist[rmkenum]{label=(\roman*),ref=\theremark\,(\roman*)}
        \crefname{rmkenumi}{Remark}{Remarks}
        
        \newlist{prpenum}{enumerate}{1} 
        \setlist[prpenum]{label=(\roman*),ref=\theproposition\,(\roman*)}
        \crefname{prpenumi}{Proposition}{Propositions}
        
        \newlist{axenum}{enumerate}{1} 
        \setlist[axenum]{label=(\roman*),ref=\theaxiom\,(\roman*)}
        \crefname{axenumi}{Axiom}{Axioms}
        
        \newlist{defcrit}{enumerate}{1} 
        \setlist[defcrit]{label=(\roman*),ref=\thedefinition\,(\roman*)}
        \crefname{defcriti}{Definition}{Definitions}
        
        \newlist{lemcrit}{enumerate}{1} 
        \setlist[lemcrit]{label=(\alph*),ref=\thelemma\,(\alph*)}
        \crefname{lemcriti}{Lemma}{Lemmas}
        
        \newlist{thmcrit}{enumerate}{1} 
        \setlist[thmcrit]{label=(\alph*),ref=\thetheorem\,(\alph*)}
        \crefname{thmcriti}{theorem}{theorems}
        
        \newlist{rmkcrit}{enumerate}{1} 
        \setlist[rmkcrit]{label=(\alph*),ref=\theremark\,(\alph*)}
        \crefname{rmkcriti}{remark}{remarks}
        
        \newlist{prpcrit}{enumerate}{1} 
        \setlist[prpcrit]{label=(\alph*),ref=\theproposition\,(\alph*)}
        \crefname{prpcriti}{proposition}{propositions}

        \makeatletter
        \let\ifnc\@ifnextchar
        \makeatother

            \reversemarginpar
            \def\redit {\marginpar{\raggedleft{See edit $\implies$}}\color{red}}
            \def\rs#1.{\redit #1.\color{black}}
            \def\rsm#1.{{\color{red} #1.}}
            
            \def\bedit {\marginpar{\raggedleft{See edit $\implies$}}\color{blue}}
            \def\bs#1.{\bedit #1.\color{black}}
            \def\bsm#1.{{\color{blue} #1.}}

            \def\ind#1_#2{\left\{#1_{#2}\right\}}
            \def\<#1>{\jp{#1}}
            \def\-#1/{{}_{#1}}
            \makeatletter
                \newcommand\pl@write[3]{%
                    \left\|#1\right\|_{L^{#2}#3}%
                    }
                \def\pl #1__#2{%
                    \def\pl@arg@i{#1}%
                    \def\pl@arg@ii{#2}%
                    \def\pl@arg@iii{\alpha}%
                    \futurelet\next\pl@eval%
                    }
                \def\pl@eval{%
                    \ifx\next\bgroup%
                            \expandafter\pl@eval@iii%
                        \else%
                            \expandafter\pl@eval@ii%
                        \fi%
                    }
                \def\pl@eval@iii#1{%
                    \pl@write\pl@arg@i\pl@arg@ii{\p{#1}}%
                    }
                \def\pl@eval@ii{%
                    \pl@write\pl@arg@i\pl@arg@ii{}%
                    }
            \makeatother
            \makeatletter
                \newcommand\ef@write[3]{%
                    ^{#1\frac{#2}{#3}}{}%
                    }
                \newcommand\Ef@write[3]{%
                    ^{#1#2/#3}{}%
                    }
                \def\ef #1{%
                    \def\ef@arg@i{}%
                    \def\ef@arg@ii{#1}%
                    \futurelet\next\ef@eval%
                    }
                \def\ief #1{%
                    \def\ef@arg@i{-}%
                    \def\ef@arg@ii{#1}%
                    \futurelet\next\ef@eval%
                    }
                \def\ef@eval{%
                    \ifx\next/%
                            \expandafter\ef@eval@iii%
                        \else%
                            \expandafter\ef@eval@ii%
                        \fi%
                    }
                \def\ef@eval@iii/{%
                    \expandafter\ef@eval@v%
                    }
                \def\ef@eval@v#1{%
                    \Ef@write\ef@arg@i\ef@arg@ii{#1}}
                \def\ef@eval@ii{%
                    \expandafter\ef@eval@iv%
                    }
                \def\ef@eval@iv#1{%
                    \ef@write\ef@arg@i\ef@arg@ii{#1}%
                    }
            \def\e@writep #1{%
                ^{+#1}{}%
                }
            \def\e@writen #1{%
                ^{-#1}{}%
                }
            \def\e #1{%
                \ifx#1-%
                        \expandafter\e@writen%
                    \else%
                        \ifx#1+%
                                \expandafter\e@writep%
                            \else%
                                ^{#1}{}%
                            \fi%
                    \fi%
                }
            \def\e@2 {%
                }

            \makeatother
            \makeatletter
            \newcommand{\undersim}[1]{\mathrel{\mathpalette\@undersim{#1}}}
            \newcommand{\@undersim}[2]{%
              \vcenter{%
                \ialign{%
                  ##\cr
                  $\m@th#1#2$\cr
                  \noalign{\nointerlineskip\kern.2ex}
                  $\m@th#1\sim$\cr
                  \noalign{\kern-.4ex}
                }%
              }%
            }
            \makeatother

    \makeatletter
    
    \def\col #1{%
        \def\col@arg{#1}%
        \begin{bmatrix}\col@i\end{bmatrix}%
        }
    \def\col@i{%
        \expandafter\col@ii\col@arg,\col@end%
        }
    \def\col@ii #1,{%
        #1%
        \futurelet\next\col@iii%
        }
    \def\col@iii{%
        \ifx\next\col@end%
                {}%
            \else%
                \\\col@ii%
            \fi%
        }
    \def\col@end{%
        {}%
        }
    \makeatother
\makeatletter\@ifundefined{remark}{\theoremstyle{remark}\newtheorem*{remark}{Remark}}{}\makeatother
\title[A null version of Anderson's rigidity theorem]{On the causality of gravitational pp-wave spacetimes: A null version of Anderson's rigidity theorem
}

\author{Hannah Cairo}
\date{September 2026}
\begin{document}
    \let\startmathold\[
    \let\endmathold\]
    \def\[{\futurelet\next\testingforaline}
    \def\aline{hello world}
    \def\testingforaline{\ifx\next\aline\expandafter\alining\else\startmathold\fi}
    \def\alining\aline#1\]{\begin{align}#1\end{align}}
\begin{abstract}
    We prove that plane waves are the only gravitational pp-wave spacetimes which are distinguishing; in particular, they are the only strongly causal ones. This is used to prove a null version of Anderson's rigidity theorem, conjectured by Costa e Silva and Flores: every geodesically complete, transversally complete, strongly causal, Ricci-flat Brinkmann spacetime of dimension $4$ is isometric to a quotient of a plane wave spacetime.
\end{abstract}
\maketitle
\section{Introduction}\label{s-introduction}
A pp-wave, i.e. a ``plane-fronted wave with parallel rays'', is a spacetime $M=(\R^4,g)$ so that in some coordinates $(u,v,x,y)$, we have
    \[\label{e-pp-wave}
        g=2dudv-H(u,x,y)du^2+dx^2+dy^2
        \]
for a smooth function $H:\R^3\to\R$. Such spacetimes were first studied by Brinkmann \cite{Brinkmann_1925} and model radiation traveling at the speed of light. When $H(u,\cdot)$ is harmonic in $(x,y)$ for every $u$, the spacetime is Ricci-flat and satisfies the vacuum Einstein field equations; such spacetimes are known as \textit{gravitational pp-waves}. When $H(u,\cdot)$ is a polynomial of degree $\leq2$ in $(x,y)$ for every $u$, $(M,g)$ is called a \textit{plane wave}. (By \cref{r-causality}(iii) below, a gravitational pp-wave which is isometric to a plane wave is itself a plane wave in this sense, so the notion does not depend on the choice of the coordinates \cref{e-pp-wave}.) The reader is encouraged to read the article \cite{Roche_2023} for a thorough discussion of pp-waves and their history.

The first result of this paper is the complete classification of the causality of gravitational pp-waves. That is, we show that plane waves are the only gravitational pp-waves which are distinguishing:
    \begin{theorem}[Causality of gravitational pp-waves]\label{t-not-distinguishing}
        Let $(M,g)$ be a gravitational pp-wave which is not a plane wave. Then $(M,g)$ is neither future-distinguishing nor past-distinguishing. Consequently, a gravitational pp-wave is distinguishing if and only if it is a plane wave.
        \end{theorem}
pp-waves are trivially causal, and the classification of the causality of gravitational plane waves was completed in 1992 by Ehrlich and Emch \cite{1992RvMaP...4..163E}: non-flat gravitational plane waves are causally continuous but not causally simple. See \cref{ss-related-work} for a discussion of the literature, and in particular of \cite{Flores_Sanchez_2003} and \cite{silva2016rigiditygeodesiccompletenessbrinkmann}, where non-distinguishing examples, and the failure of strong causality for $u$-independent $H$, were found previously in certain cases. The general case remained open until now.

Our second result concerns Brinkmann spacetimes, i.e. connected time-oriented Lorentzian manifolds $(M,g)$ admitting a complete parallel null vector field $X$; pp-waves are the special case $X=\partial_v$. Following \cite{silva2016rigiditygeodesiccompletenessbrinkmann}, we say that a Brinkmann spacetime is \textit{transversally complete} if there is a complete vector field $Y$ with $g(X,Y)=1$ and $[X,Y]=0$. Anderson \cite{Anderson_2000} proved that every geodesically complete, chronological, Ricci-flat stationary spacetime of dimension $4$ is isometric to a quotient of Minkowski space. Regarding Brinkmann spacetimes as null analogues of stationary spacetimes, Costa e Silva and Flores conjectured a null version of this theorem: every strongly causal, Ricci-flat Brinkmann spacetime of dimension $4$ satisfying a suitable completeness condition should be isometric to a quotient of a plane wave (\cite[Conjecture~VI.2]{silva2016splittingproblemlorentzianmanifolds}, \cite[Conjecture~1.1]{silva2016rigiditygeodesiccompletenessbrinkmann}). Together with Herrera, they proved this, with geodesic completeness as the completeness condition, under the additional assumption that the field $Y$ conjugate to $X$ can be chosen to be a Killing field \cite[Theorem~1.2]{silva2016rigiditygeodesiccompletenessbrinkmann}; they also observed that, if the completeness condition is taken to be geodesic completeness together with transversal completeness, then the conjecture reduces to the Ehlers--Kundt conjecture \cite{Ehlers_Kundt_1962}, which asserts that every geodesically complete gravitational pp-wave is a plane wave. The Ehlers--Kundt conjecture is false \cite{Cairo_EK}; nevertheless, the null version of Anderson's theorem is true, because the causality hypothesis alone already forces a gravitational pp-wave to be a plane wave. From \cref{t-not-distinguishing} we deduce:
    \begin{theorem}[Null version of Anderson's rigidity theorem]\label{t-conj-flores-silva}
        Every geodesically complete, transversally complete, strongly causal, Ricci-flat Brinkmann spacetime of dimension $4$ is isometric to a quotient of a plane wave spacetime.
    \end{theorem}
See \cref{s-causality-theorem} for the proof of \cref{t-not-distinguishing}, and \cref{ss-anderson} for the proof of \cref{t-conj-flores-silva}. Geodesic completeness enters the latter only through a theorem of \cite{silva2016rigiditygeodesiccompletenessbrinkmann} which identifies the universal cover with a pp-wave; see \cref{r-completeness}.

\Cref{t-not-distinguishing} may also be read as recovering something of what Ehlers and Kundt presumably had in mind. Their conjecture was that a gravitational pp-wave which is physically reasonable, with geodesic completeness serving as the measure of physical reasonableness, must be a plane wave. With that measure the conjecture fails \cite{Cairo_EK}; but if physical reasonableness is measured instead by causality, then the corresponding statement is true, with no completeness assumption at all: a gravitational pp-wave which is distinguishing (let alone strongly causal) is a plane wave. Whether the distinguishing property should itself be counted as a physical requirement is not beyond debate. Hubeny, Rangamani and Ross \cite{Hubeny_Rangamani_Ross_2005} have argued, using examples from string theory, that spacetimes which are merely causal, and hence arbitrarily close to admitting closed causal curves, can still be physically relevant; the non-plane gravitational pp-waves of \cref{t-not-distinguishing} are all of this kind.
\subsection{Relation to previous work}\label{ss-related-work}
The causality of pp-waves has been studied extensively; see \cite{Flores-on-the-geometry-of-pp-wave-spacetimes,Roche_2023} and the review \cite{GarciaParrado_Senovilla_2005}. Penrose \cite{Penrose1965} observed that plane waves are not globally hyperbolic, and Ehrlich and Emch \cite{1992RvMaP...4..163E} determined their position in the causal ladder. For general pp-wave type spacetimes, Flores and S\'anchez \cite{Flores_Sanchez_2003} showed that the growth of $H$ at spatial infinity is decisive (we state their results in our sign convention, $g=2dudv-Hdu^2+dx^2+dy^2$, in which our $H$ is their $-H$): if $H(u,\cdot)$ is bounded above by a subquadratic function $R_1(u)|x|^{2-\eps}$ at infinity (and the wave fronts are complete, as they are here), then the spacetime is globally hyperbolic; if $H(u,\cdot)\leq R_1(u)|x|^2$ at infinity, then it is strongly causal; and these spacetimes are always causal, but superquadratic growth, even of order $|x|^{2+\eps}$, may destroy strong causality or even the distinguishing property (see also \cite{Hubeny_Rangamani_2002,Flores_Sanchez_2008}). That some pp-waves fail to be distinguishing was also noticed in \cite{Hubeny_Rangamani_Ross_2004}.

Minguzzi \cite{Minguzzi_2012} translated the causal ladder of spacetimes admitting a parallel null vector into properties of the least action of the associated classical mechanical system (in our notation, the action of a curve $\gamma$ is $\tfrac12\p{\mathcal K(\gamma)-\mathcal H(\gamma)}$; cf.\ \cref{causality-lemma} below). In particular, he showed that for such spacetimes strong causality is equivalent to stable causality, and to the lower semicontinuity of the least action on the diagonal, and he gave examples in which the distinguishing property fails. In the gravitational case, none of these results apply to non-plane waves; harmonicity implies that $H(u,\cdot)$ is either a polynomial of degree $\leq2$ or not bounded above or below by any quadratic polynomial.

Closest to \cref{t-not-distinguishing} is the work of Costa e Silva, Flores and Herrera \cite{silva2016rigiditygeodesiccompletenessbrinkmann}: the proof of their Theorem~1.2 shows that a gravitational pp-wave with $u$-independent, non-quadratic $H$ is not strongly causal. \Cref{t-not-distinguishing} removes the assumption that $H$ is independent of $u$ and strengthens the conclusion from the failure of strong causality to the failure of the (weaker) distinguishing condition. Our proof of \cref{causality-main-lemma} is also different: instead of estimating the integral of $H$ along radial paths, we consider it as a function of the direction $\theta$ of the excursion and use its Fourier expansion in $\theta$, which allows $H$ to vary with $u$.
 
One consequence is worth making explicit. By \cref{t-not-distinguishing}, a gravitational pp-wave which is distinguishing (in particular, one which is strongly causal) is a plane wave, \textit{with no completeness assumption}. Hence within the class of pp-waves, the causality hypothesis in \cite{silva2016rigiditygeodesiccompletenessbrinkmann} and in the conjecture of \cite{silva2016splittingproblemlorentzianmanifolds} already suffices on its own; conversely, every counterexample to the Ehlers--Kundt conjecture, such as the ones constructed in \cite{Cairo_EK}, necessarily fails to be distinguishing. The counterexamples of \cite{Cairo_EK} have $u$-independent $H$, so that $\partial_u$ is a Killing field and they are transversally Killing in the sense of \cite{silva2016rigiditygeodesiccompletenessbrinkmann}; they are consistent with \cite[Theorem~1.2]{silva2016rigiditygeodesiccompletenessbrinkmann} only because they fail to be strongly causal.

\subsection{Statement on AI usage}
All ideas in this paper are of human origin. The author used Claude Fable 5.1 and Claude Sonnet 5 to generate figures, check proofs, search the literature, edit, and write parts of this paper.

\subsection{Notation}\label{ss-notation}
Throughout the proof, we will make use of many different constants that depend on each other in different ways.
    \begin{definition}~\\
        \begin{itemize}
        \item[$(\lesssim)$] Suppose that $A,B>0$ are two positive quantities. If there exists some constant $C$ so that $A\leq CB$, then we write $A\lesssim B$ or $B\gtrsim A$. If $A\lesssim B$ and $B\lesssim A$, then we write $A\sim B$. If the constant in $A\leq CB$ depends on other parameters, e.g. $D$, then we will display this dependence as $A\lesssim_DB$.
        \item[$(\ll)$] Let $P(x,y)$ be a proposition about two positive quantities $x,y$. Suppose that there exists some constant $c>0$ such that $x\leq cy\implies P(x,y)$. Then we write this statement as follows:
            \begin{enumerate}
                \item[$(\bullet)$] If $x\ll y$, then $P(x,y)$.
                \end{enumerate}
            We define $\gg$ similarly (i.e. $x\gg y$ means $y\ll x$). When we are dealing with an arbitrarily small positive number $x$, we often write $x\ll1$. We also write $x\ll0$ to mean that $x\leq-C$ for some sufficiently large constant $C$; that is, $x$ is a large negative number, and similarly $x\gg0$ to mean that $x\geq C$ for some sufficiently large constant $C$. (These are not covered by the definition above, which requires $y>0$.)
        \item[$(D_R)$] For $z_0\in\C$ and $R>0$ we write $D_R(z_0):=\{z\in\C:|z-z_0|<R\}$ for the open disc of radius $R$ about $z_0$, and $D_R:=D_R(0)$.
        \end{itemize}
    \end{definition}

\section{The complete classification of causality for gravitational pp-wave spacetimes}\label{s-causality-theorem}

Let $(M,g)$ be a gravitational pp-wave with metric
    \[g=2dudv-H(u,x,y)du^2+dx^2+dy^2\]
as in \cref{s-introduction}, where $H$ is smooth (continuity of $H$ is all that we use) and harmonic in $(x,y)$. Let $\bm\sigma:[t_0,t_1]\to M$ be some curve, and write $\bm\sigma=(\bm u,\bm v,\bm x,\bm y)$. Note that $\bm\sigma$ is a timelike curve if and only if we have $g(\dot{\bm\sigma},\dot{\bm\sigma})<0$. i.e.
    \[
        2\dot{\bm v}\dot{\bm u}<H(\bm u,\bm x,\bm y)\dot{\bm u}^2-\dot{\bm x}^2-\dot{\bm y}^2
        \]
This necessarily implies that either $\dot{\bm u}>0$ everywhere or $\dot{\bm u}<0$ everywhere. We call the first case the future-directed case and the second the past-directed case.

Let
    \[N:=\{u\in\R:H(u,\cdot)\text{ is not a polynomial of degree}\leq2\text{ in }(x,y)\}.\]
Since $H$ is continuous, $N$ is open (if $u_k\to u$ and each $H(u_k,\cdot)$ is a polynomial of degree $\leq2$, then so is the locally uniform limit $H(u,\cdot)$, because these polynomials form a finite-dimensional, hence closed, subspace). By definition, $N$ is nonempty if and only if $(M,g)$ is not a plane wave; by \cref{r-causality}(iii) below, this is the same as $(M,g)$ not being isometric to any plane wave. We would like to show the following:
\begin{theorem}\label{t-causality}
    Let $(M,g)$ be any gravitational pp-wave as above with $N\neq\emptyset$, and let $p=(u_0,v_0,x_0,y_0)\in M$.
    \begin{enumerate}
        \item[(a)] Suppose that $N\cap(u_0,u_0+\eps)\neq\emptyset$ for every $\eps>0$. Then, given any precompact neighborhood $U$ of $p$ and any neighborhood $V\subseteq U$ of $p$, there is some curve $\bm\sigma:[t_0,t_1]\to M$ that satisfies the following:
       \begin{enumerate}
           \item[(1)] $\bm\sigma$ is a future-directed timelike curve.
           \item[(2)] $\bm\sigma(t_0)=p$.
           \item[(3)] $\bm\sigma(t_1)\in V$.
           \item[(4)] The image of $\bm\sigma$ is not entirely contained in $U$.
       \end{enumerate}
        \item[(b)] Suppose that $N\cap(u_0-\eps,u_0)\neq\emptyset$ for every $\eps>0$. Then the same holds with ``future-directed'' replaced by ``past-directed''.
    \end{enumerate}
    In particular, both (a) and (b) apply whenever $u_0\in N$, and such points exist.
\end{theorem}

\vspace{-1cm}

\begin{remark}\label{r-causality}~\\
    \begin{enumerate}
        \item[(i)] \Cref{t-causality} implies that $(M,g)$ is neither future nor past distinguishing (and in particular not distinguishing). Indeed, a spacetime is future distinguishing at $p$ if and only if for every neighborhood $U$ of $p$ there is a neighborhood $V\subseteq U$ of $p$ such that for every future-directed causal curve $\bm\sigma:[t_0,t_1]\to M$ starting at $p$, the set $\bm\sigma\inv(V)$ is an interval (see e.g. \cite{Roche_2023} and the references therein).

The curve $\bm\sigma$ of \cref{t-causality}(a) starts at $p\in V$, leaves $U\supseteq V$, and returns to $V$, so $\bm\sigma\inv(V)$ is not an interval; so $(M,g)$ is not future distinguishing at any $p$ as in (a), e.g. at any $p$ with $u_0\in N$. The past case is identical, using (b).
        \item[(ii)] Suppose that no neighborhood of $p$ is isometric to a neighborhood of a plane wave spacetime. If $H(u,\cdot)$ were a quadratic polynomial in $(x,y)$ for all $u$ in some open interval $J\ni u_0$, then the open set $\{u\in J\}$ would be a plane wave neighborhood of $p$. Hence $u_0\in\overline N$, so that the hypothesis of (a) or the hypothesis of (b) is satisfied at $p$. Note, however, that it can happen that only one of them is satisfied: if $H(u,\cdot)$ is quadratic for all $u\geq u_0$ but $u_0\in\overline N$, then the hypothesis of (a) fails at $p$, and only (b) applies at $p$. (In this case the conclusion of (a) fails as well: every future-directed causal curve starting at $p$ stays in the region $\{u\geq u_0\}$, where the metric agrees with that of a plane wave, which is strongly causal.)
        \item[(iii)] Plane waves are strongly causal, and in particular future and past distinguishing (this is the case of at most quadratic growth in \cite{Flores_Sanchez_2003}), and an isometry either preserves or exchanges the two time orientations. Hence (i) shows that a gravitational pp-wave with $N\neq\emptyset$ is not isometric to any plane wave. (This can also be seen from a curvature computation, cf.\ \cite{Cairo_EK}, but we do not need this.) In particular, the notion of a plane wave introduced in \cref{s-introduction} does not depend on the choice of coordinates: if a gravitational pp-wave is isometric to a plane wave, then $N=\emptyset$.
    \end{enumerate}
\end{remark}

Part (b) follows from part (a) applied to the pp-wave with $H(u,x,y)$ replaced by $H(-u,x,y)$, since $(u,v,x,y)\mapsto(-u,-v,x,y)$ is an isometry between the two which exchanges future-directed and past-directed curves. So we only need to prove (a). To do so, we note that there are some simplifications that we can make. First, since $\dot{\bm u}>0$ everywhere, we may reparameterize the curve by $u=\bm u(t)$. In general, we write $\bm\sigma(u)=(u,\bm v(u),\gamma(u))$, for some function $\gamma$. As it turns out, the function $\bm v(u)$ will not be very important, as \cref{causality-lemma} below shows. We first define the relevant functionals. Throughout this section we identify $\R^2$ with $\C$, and we write $H(u,x,y)=\Re\p{\mathcal f(u;x+iy)}$ for a function $\mathcal f(u;z)$ which is entire in $z$ for each $u$ and continuous in $(u,z)$; this is possible since $H(u,\cdot)$ is harmonic.

\begin{definition}
    Let $\gamma:[u_0,u_1]\to\R^2$ be any (piecewise $C^1$) curve. We define the following functionals:
        \[\aline
            \mathcal I(\gamma)&=\int_{u_0}^{u_1}\mathcal f(u;\gamma(u))du\\
            \mathcal H(\gamma)&=\int_{u_0}^{u_1}H(u,\gamma(u))du\\
            \mathcal K(\gamma)&=\int_{u_0}^{u_1}|\dot\gamma(u)|^2du
            \]
    (The functional $\mathcal I$ will only be used in the proof of \cref{causality-main-lemma}; note that $\mathcal H=\Re\p{\mathcal I}$.)
\end{definition}

\begin{lemma}\label{causality-lemma}
    Let $\gamma:[u_0,u_1]\to\R^2$ with $\gamma(u_0)=(x_0,y_0)$ and $\gamma(u_1)=(x_1,y_1)$. Then there exists a future-directed timelike curve from $(u_0,v_0,x_0,y_0)$ to $(u_1,v_1,x_1,y_1)$ whose $(u,x,y)$-projection is a reparameterization of $\gamma$ iff $u_0<u_1$ and
        \[2(v_1-v_0)<\mathcal H(\gamma)-\mathcal K(\gamma)\]
\end{lemma}

\begin{proof}
    Parameterizing by $u$, the timelike condition reads $2\bm v'(u)<H(u,\gamma(u))-|\gamma'(u)|^2$. If such a $\bm v$ exists then integrating gives the inequality; conversely, if the inequality holds, then $\bm v(u):=v_0+\tfrac12\int_{u_0}^u\big(H(s,\gamma(s))-|\gamma'(s)|^2\big)ds-\delta(u-u_0)$ works for a suitable $\delta>0$ (namely the one for which $\bm v(u_1)=v_1$).
\end{proof}

To show \cref{t-causality}, we will use dilated and rotated copies of a fixed closed curve. For a closed curve $\gamma$ based at $z_0=x_0+iy_0$, and for $R>0$ and $\theta\in\R$, we denote by
    \[\gamma^{R,\theta}:=z_0+Re^{i\theta}(\gamma-z_0)\]
the curve obtained by dilating and rotating $\gamma$ about its base point. The key step is the following.

\begin{lemma}\label{causality-main-lemma}
    Let $[u_0,u_1]$ be an interval such that $H(u,\cdot)$ is not a quadratic polynomial in $(x,y)$ for some $u\in[u_0,u_1]$, and let $z_0=x_0+iy_0\in\C$; equivalently (writing $\mathcal f(u;z)=\sum_ma_m(u)(z-z_0)^m$), some coefficient $a_n$ with $n\geq3$ does not vanish identically on $[u_0,u_1]$. Fix such an $n$. Then there exists a smooth closed curve $\gamma:[u_0,u_1]\to\R^2$ with $\gamma(u_0)=\gamma(u_1)=(x_0,y_0)$ such that, for any $R\gg1$ sufficiently large (depending on $\gamma$), there exists some $\theta\in[0,2\pi)$ so that
        \[\aline
            \mathcal H(\gamma^{R,\theta})&\gtrsim_\gamma R^n\\
            \mathcal K(\gamma^{R,\theta})&\lesssim_\gamma R^2
            \]
\end{lemma}

To see that \cref{causality-main-lemma} implies \cref{t-causality}(a), let $p=(u_0,v_0,x_0,y_0)$ be as in (a). Then every interval $[u_0,u_1]$ with $u_1>u_0$ contains some $u$ for which $H(u,\cdot)$ is not quadratic. We apply \cref{causality-main-lemma} on $[u_0,u_1]$ with $u_1-u_0$ so small that $(u_1,v_0,x_0,y_0)\in V$, and then \cref{causality-lemma} with $v_1=v_0$ to the curve $\gamma^{R,\theta}$ (this is permitted since $\mathcal H(\gamma^{R,\theta})-\mathcal K(\gamma^{R,\theta})\gtrsim_\gamma R^n-O_\gamma(R^2)>0$ for $R$ large, as $n\geq3$). We obtain a future-directed timelike curve from $p$ to $(u_1,v_0,x_0,y_0)\in V$ whose $(x,y)$-projection is $\gamma^{R,\theta}$. Since $\gamma$ is not constant, this projection leaves any given compact subset of $\R^2$ once $R$ is large, so the curve leaves the precompact set $U$. This proves \cref{t-causality}(a).

The argument is illustrated in \cref{fig-causality}.

\begin{figure}[htbp]
    \centering
    \includegraphics[width=\textwidth]{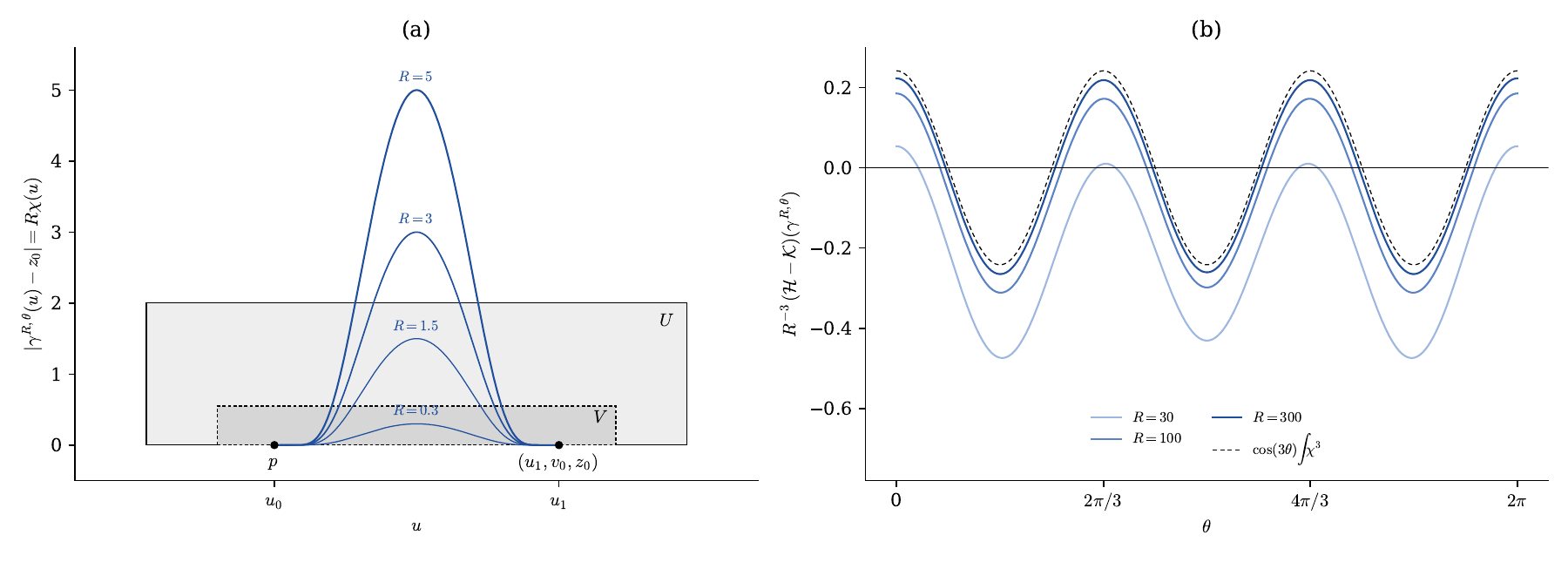}
    \caption{(a) A schematic picture of the curves $\gamma^{R,\theta}$ for $\gamma=\chi\geq0$ a bump function (as in the proof of \cref{causality-main-lemma}): the curve makes a radial excursion of size $\sim R$ in the direction $e^{i\theta}$ and returns to $z_0$ at time $u_1$. For $R$ large it leaves $U$, but its endpoints remain in $V$. (b) The function $\theta\mapsto R^{-3}(\mathcal H-\mathcal K)(\gamma^{R,\theta})$ for the example $\mathcal f(z)=z^3+3z^2+z$ (so $H=\Re\p{\mathcal f}$ and $n=3$), $z_0=0$, $[u_0,u_1]=[0,1]$ and $\chi(u)=e^{4-1/(u(1-u))}$. As $R\to\infty$ it converges to the $n$th Fourier mode $\cos(3\theta)\int\chi^3$; wherever it is positive, \cref{causality-lemma} yields a timelike curve over $\gamma^{R,\theta}$ with $v_1=v_0$. In this example that first happens (near $\theta=0$) for $R\approx23.3$, and for $R\gtrsim28.7$ it happens near all three maxima of $\cos(3\theta)$.}
    \label{fig-causality}
\end{figure}

\begin{proof}[Proof of \cref{causality-main-lemma}]
    The statement about $\mathcal K$ is trivial, since $\mathcal K(\gamma^{R,\theta})=R^2\mathcal K(\gamma)$. To show the one about $\mathcal H$, we assume $z_0=0$ without loss of generality (replace $\mathcal f(u;z)$ by $\mathcal f(u;z+z_0)$), so that $\gamma^{R,\theta}=Re^{i\theta}\gamma$, and write
        \[\mathcal f(u;z)=\sum_{m=0}^\infty a_m(u)z^m\]
    which converges for all $z$ since $\mathcal f(u;z)$ is holomorphic on all of $\C$ for any $u$; the coefficients $a_m(u)=\frac1{2\pi i}\int_{\partial D_1}\mathcal f(u;\xi)\xi^{-m-1}d\xi$ are continuous in $u$. Recall that $\mathcal H=\Re\p{\mathcal I}$. The idea is to consider $\theta\mapsto\mathcal H(Re^{i\theta}\gamma)$ as a function on the circle, and to show that its $n$th Fourier coefficient is $\tfrac12R^n\int a_n\gamma^n$ while its mean does not depend on $R$. For a suitable choice of $\gamma$ the former does not vanish, and this forces $\max_\theta\mathcal H(Re^{i\theta}\gamma)\gtrsim_\gamma R^n-O_\gamma(1)$.

    For any $\gamma$, let $I_\gamma(\theta)=\mathcal I(e^{i\theta}\gamma)$ and also write $\gamma(u)=\rho(u)e^{i\alpha(u)}$. For $m\in\Z$, the Fourier coefficients of $I_\gamma$ are
        \begin{align}
            \hat I_\gamma(m)
                &=\frac1{2\pi}\int_0^{2\pi}\int_{u_0}^{u_1}\mathcal f(u;e^{i\theta}\gamma(u))e^{-im\theta}dud\theta\\
                &=\frac1{2\pi i}\int_{u_0}^{u_1}\int_{S^1}\frac{\mathcal f(u;\rho(u)e^{i\alpha(u)}\zeta)}{\zeta^{m+1}}d\zeta du\\
                &=\frac1{2\pi i}\int_{u_0}^{u_1}(\gamma(u))^m\int_{\partial D_{\rho(u)}}\frac{\mathcal f(u;\xi)}{\xi^{m+1}}d\xi du\\
                &=\frac1{m!}\int_{u_0}^{u_1}(\gamma(u))^m\partial_z^m\mathcal f(u;0)du\\
                &=\int_{u_0}^{u_1}a_m(u)(\gamma(u))^mdu
            \end{align}
    for $m\geq0$ (where $\gamma(u)=0$, the inner integrals are to be read as $a_0(u)$ if $m=0$ and as $0$ otherwise), and the same computation shows that $\hat I_\gamma(m)=0$ whenever $m<0$.

    Ultimately, this comes from the fact that Cauchy's integral formula is a Fourier transform in disguise: the Fourier coefficients of the restriction of an entire function to a circle are, up to scaling, its Taylor coefficients, and those of negative order vanish.

    Note that $H_\gamma(\theta):=\mathcal H(e^{i\theta}\gamma)=\Re\p{I_\gamma(\theta)}$ satisfies
        \[\label{fourier-coefficients-H}
            \hat H_\gamma(m)=\frac{\hat I_\gamma(m)+\overline{\hat I_\gamma(-m)}}2=
                \begin{cases}
                \tfrac12\hat I_\gamma(m),& \text{if $m>0$}\\
                \Re\hat I_\gamma(0),& \text{if $m=0$}
                \end{cases}
            \]
    By the triangle inequality, $\pl H_\gamma__1\gtrsim\pl\hat H_\gamma__\infty$, so for every $m>0$, we have
        \[\label{H-bound-L1}
            \int_0^{2\pi}|H_\gamma(\theta)|d\theta\gtrsim|\hat I_\gamma(m)|=\abs{\int_{u_0}^{u_1}a_m(u)(\gamma(u))^mdu}
            \]
    Since $\int_0^{2\pi}H_\gamma(\theta)d\theta=2\pi\Re\p{\hat I_\gamma(0)}=2\pi\Re\p{\int_{u_0}^{u_1}a_0(u)du}$, we deduce that
        \[\label{H-int-positive}
            \int_0^{2\pi}H_\gamma^+d\theta=\tfrac12\p{\int_0^{2\pi}H_\gamma d\theta+\int_0^{2\pi}|H_\gamma|d\theta}\gtrsim|\hat I_\gamma(m)|-O(|\hat I_\gamma(0)|)\]
    where $h^+:=\frac{h+|h|}2$ denotes the positive part of a real-valued function $h$.

    We now take $m=n$. Note that $\hat I_{R\gamma}(n)=R^n\hat I_\gamma(n)$, while $\hat I_{R\gamma}(0)=\int_{u_0}^{u_1}a_0(u)du$ does not depend on $R$ or $\gamma$. It follows that
        \[|\hat I_{R\gamma}(n)|\sim_{\gamma}R^n\]
    provided that
        \[\label{e-gamma-nondegenerate}
            \int_{u_0}^{u_1}a_n(u)(\gamma(u))^ndu\neq0\]
    It remains to choose $\gamma$ so that \cref{e-gamma-nondegenerate} holds. By assumption, $a_n$ is not identically zero on $[u_0,u_1]$ (if $\mathcal f(u;\cdot)$ is not quadratic for some $u\in[u_0,u_1]$, then $a_m(u)\neq0$ for some $m\geq3$; this is the equivalence claimed in the statement). Since $a_n$ is continuous, there is some $u_*\in(u_0,u_1)$ and a compact interval $J\subset(u_0,u_1)$ with nonempty interior containing $u_*$ such that $|a_n(u)-a_n(u_*)|\leq\frac12|a_n(u_*)|$ for $u\in J$, and $a_n(u_*)\neq0$. Let $\gamma(u):=\chi(u)\geq0$ be a smooth real-valued bump function, not identically zero, which is supported in $J$. Then $\gamma(u_0)=\gamma(u_1)=0$ and
        \[\abs{\int_{u_0}^{u_1}a_n\gamma^ndu-a_n(u_*)\int_J\chi^ndu}\leq\tfrac12|a_n(u_*)|\int_J\chi^ndu,\]
    so \cref{e-gamma-nondegenerate} holds. (In fact, \cref{e-gamma-nondegenerate} holds for $\gamma$ in an open and dense subset of the closed curves based at $0$: the left-hand side is a continuous homogeneous polynomial in $\gamma$ which, as we have just seen, does not vanish identically.)

    Finally, let $\theta_R$ be a value of $\theta\in[0,2\pi)$ which maximizes $H_{R\gamma}(\theta)$. Since $2\pi\max_\theta H_{R\gamma}\geq\int_0^{2\pi}H_{R\gamma}^+d\theta$, \cref{H-int-positive} gives
        \[\mathcal H(Re^{i\theta_R}\gamma)=H_{R\gamma}(\theta_R)\gtrsim_{\gamma} R^n-O(1)\gtrsim_\gamma R^n\]
    for $R$ sufficiently large. This completes the proof.
    \end{proof}

\section{Proof of \texorpdfstring{\cref{t-conj-flores-silva}}{Theorem \ref{t-conj-flores-silva}}}\label{ss-anderson}

Let $(M,g)$ be a geodesically complete, transversally complete, strongly causal, Ricci-flat Brinkmann spacetime of dimension $4$, with parallel null field $X$ and conjugate complete field $Y$, and let $\pi:(\tilde M,\tilde g)\to(M,g)$ be its universal cover, with the pulled back metric and time orientation. Then $(\tilde M,\tilde g)$ is again a Brinkmann spacetime which is geodesically complete, transversally complete and Ricci-flat: geodesics and integral curves of $X$ and $Y$ lift to $\tilde M$, so the lifts $\tilde X$ and $\tilde Y$ of $X$ and $Y$ are complete, and the identities $\tilde\nabla\tilde X=0$, $\tilde g(\tilde X,\tilde Y)=1$, $[\tilde X,\tilde Y]=0$ and $\widetilde{\operatorname{Ric}}=0$ are local. Moreover, $(\tilde M,\tilde g)$ is strongly causal. Indeed, let $\tilde p\in\tilde M$ and let $\tilde U$ be any neighborhood of $\tilde p$; shrinking $\tilde U$, we may assume that $\pi$ maps $\tilde U$ diffeomorphically onto $U:=\pi(\tilde U)$ and that $\tilde U$ is a connected component of $\pi\inv(U)$. By strong causality of $(M,g)$ at $\pi(\tilde p)$, there is a neighborhood $V\subseteq U$ of $\pi(\tilde p)$ such that every causal curve of $M$ with both endpoints in $V$ is contained in $U$. If $\tilde\sigma$ is a causal curve of $\tilde M$ with both endpoints in $\tilde V:=\tilde U\cap\pi\inv(V)$, then $\pi\circ\tilde\sigma$ is a causal curve of $M$ with endpoints in $V$, so it is contained in $U$; hence $\tilde\sigma$ is a connected subset of $\pi\inv(U)$ which meets the component $\tilde U$, so $\tilde\sigma\subseteq\tilde U$. This shows that $(\tilde M,\tilde g)$ is strongly causal at $\tilde p$.

By \cite[Theorem~2.1]{silva2016rigiditygeodesiccompletenessbrinkmann} (which is \cite[Theorem~V.11]{silva2016splittingproblemlorentzianmanifolds}), applied to the strongly causal and transversally complete spacetime $(M,g)$, the universal cover $(\tilde M,\tilde g)$ is isometric to a standard Brinkmann spacetime in the sense of \cite[Section~2]{silva2016rigiditygeodesiccompletenessbrinkmann}. Since $\tilde M$ is simply connected, $4$-dimensional, geodesically complete and Ricci-flat, \cite[Theorem~3.1]{silva2016rigiditygeodesiccompletenessbrinkmann} shows that $(\tilde M,\tilde g)$ is isometric to a pp-wave $(\R^4,g)$ of the form \cref{e-pp-wave} with $H$ harmonic in $(x,y)$, i.e. to a gravitational pp-wave. This gravitational pp-wave is strongly causal, and in particular distinguishing, so it is a plane wave by \cref{t-not-distinguishing}. Finally, $(M,g)$ is isometric to the quotient of $(\tilde M,\tilde g)$ by the group of deck transformations, which acts freely and properly discontinuously by isometries. This proves \cref{t-conj-flores-silva}.\qed

\begin{remark}\label{r-completeness}
    Geodesic completeness of $(M,g)$ is used above only through \cite[Theorem~3.1]{silva2016rigiditygeodesiccompletenessbrinkmann}, where it guarantees that the universal cover is a pp-wave on all of $\R^4$ (rather than on $\R^2\times Q$ for a possibly incomplete flat surface $Q$). For pp-waves in the sense of \cref{e-pp-wave}, no completeness assumption is needed, by \cref{t-not-distinguishing}. Some completeness assumption on the wave fronts is nevertheless necessary. Indeed, let $D\subset\R^2$ be the open unit disk and let $M=\R^2\times D$ with
        \[g=2dudv-\Re\p{(x+iy)^3}du^2+dx^2+dy^2.\]
    This is a Ricci-flat Brinkmann spacetime of dimension $4$; it is transversally complete, with $X=\partial_v$ and $Y=\partial_u$; and it is simply connected but not locally isometric to a plane wave near any point, since $\mathcal f(z)=z^3$ has $\mathcal f'''\neq0$ (cf.\ \cite{Cairo_EK}). It is, however, strongly causal, because $H$ is bounded on $D$: if $|H|\leq C$ and $\bm\sigma$ is a future-directed causal curve parameterized by $u\in[u_0,u_0+\delta]$, then $2\bm v'\leq H-|\gamma'|^2\leq C$, so $\bm v$ can never rise more than $C\delta/2$ above $\bm v(u_0)$, and if it ever drops by more than $C\delta$ it cannot return; hence a causal curve with both endpoints near $p=\bm\sigma(u_0)$ satisfies $|\bm v-v_0|\lesssim C\delta$, and then $\int|\gamma'|^2du\leq\int H\,du-2(\bm v(u_0+\delta)-v_0)\lesssim C\delta$, so that its $(x,y)$-excursion is $O(\delta)$ by the Cauchy--Schwarz inequality. (This is the mechanism behind the ``at most quadratic growth'' criterion of \cite{Flores_Sanchez_2003}; completeness of the wave fronts enters there only in the global hyperbolicity statement. See also \cite[Proposition~1.1]{Minguzzi_2012}, where a potential which is bounded above is shown to give a time function, hence stable causality.) Thus $(M,g)$ satisfies all hypotheses of \cref{t-conj-flores-silva} except geodesic completeness, and is not a quotient of a plane wave; the integral curves of $X$ and $Y$ are complete, so incompleteness of the wave fronts alone suffices to break the theorem.
\end{remark}

\printbibliography

@misc{Cairo_EK,
  author        = {Cairo, Hannah},
  title         = {On the completeness of gravitational pp-wave spacetimes: a counterexample to the {E}hlers--{K}undt conjecture},
  year          = {2026},
  note          = {Preprint},
  eprint        = {XXXX.XXXXX},
  archivePrefix = {arXiv}
}

@article{Brinkmann_1925,
  author  = {Brinkmann, H. W.},
  title   = {{E}instein spaces which are mapped conformally on each other},
  journal = {Mathematische Annalen},
  volume  = {94},
  number  = {1},
  pages   = {119--145},
  year    = {1925},
  doi     = {10.1007/BF01208647}
}

@incollection{Ehlers_Kundt_1962,
  author    = {Ehlers, J{\"u}rgen and Kundt, Wolfgang},
  title     = {Exact solutions of the gravitational field equations},
  booktitle = {Gravitation: an introduction to current research},
  editor    = {Witten, Louis},
  publisher = {John Wiley \& Sons},
  address   = {New York},
  year      = {1962},
  pages     = {49--101}
}

@article{Roche_2023,
  author  = {Roche, Cian and Aazami, Amir Babak and Cederbaum, Carla},
  title   = {Exact parallel waves in general relativity},
  journal = {General Relativity and Gravitation},
  volume  = {55},
  number  = {2},
  pages   = {40},
  year    = {2023},
  doi     = {10.1007/s10714-023-03083-x},
  eprint  = {2207.03591},
  archivePrefix = {arXiv}
}

@article{silva2016rigiditygeodesiccompletenessbrinkmann,
  author  = {Costa e Silva, Ivan P. and Flores, Jos{\'e} Luis and Herrera, J{\'o}natan},
  title   = {Rigidity of geodesic completeness in the {B}rinkmann class of gravitational wave spacetimes},
  journal = {Advances in Theoretical and Mathematical Physics},
  volume  = {22},
  number  = {1},
  pages   = {25--45},
  year    = {2018},
  doi     = {10.4310/ATMP.2018.v22.n1.a2},
  eprint  = {1605.03619},
  archivePrefix = {arXiv}
}

@article{silva2016splittingproblemlorentzianmanifolds,
  author  = {Costa e Silva, Ivan P. and Flores, Jos{\'e} Luis},
  title   = {On the splitting problem for {L}orentzian manifolds with an {$\mathbb{R}$}-action with causal orbits},
  journal = {Annales Henri Poincar{\'e}},
  volume  = {18},
  number  = {5},
  pages   = {1635--1670},
  year    = {2017},
  doi     = {10.1007/s00023-017-0551-8},
  eprint  = {1605.02345},
  archivePrefix = {arXiv}
}

@article{Anderson_2000,
  author  = {Anderson, Michael T.},
  title   = {On stationary vacuum solutions to the {E}instein equations},
  journal = {Annales Henri Poincar{\'e}},
  volume  = {1},
  number  = {5},
  pages   = {977--994},
  year    = {2000}
}

@article{Penrose1965,
  author  = {Penrose, Roger},
  title   = {A remarkable property of plane waves in general relativity},
  journal = {Reviews of Modern Physics},
  volume  = {37},
  pages   = {215--220},
  year    = {1965},
  doi     = {10.1103/RevModPhys.37.215}
}

@article{1992RvMaP...4..163E,
  author  = {Ehrlich, Paul E. and Emch, G{\'e}rard G.},
  title   = {Gravitational waves and causality},
  journal = {Reviews in Mathematical Physics},
  volume  = {4},
  number  = {2},
  pages   = {163--221},
  year    = {1992},
  doi     = {10.1142/S0129055X92000066},
  note    = {Erratum: Rev. Math. Phys. 4 (1992), 501}
}

@article{Flores_Sanchez_2003,
  author  = {Flores, Jos{\'e} L. and S{\'a}nchez, Miguel},
  title   = {Causality and conjugate points in general plane waves},
  journal = {Classical and Quantum Gravity},
  volume  = {20},
  number  = {11},
  pages   = {2275--2291},
  year    = {2003},
  doi     = {10.1088/0264-9381/20/11/322},
  eprint  = {gr-qc/0211086},
  archivePrefix = {arXiv}
}

@incollection{Flores-on-the-geometry-of-pp-wave-spacetimes,
  author    = {Flores, Jos{\'e} L. and S{\'a}nchez, Miguel},
  title     = {On the geometry of pp-wave type spacetimes},
  booktitle = {Analytical and numerical approaches to mathematical relativity},
  editor    = {Frauendiener, J{\"o}rg and Giulini, Domenico J. W. and Perlick, Volker},
  series    = {Lecture Notes in Physics},
  volume    = {692},
  publisher = {Springer},
  address   = {Berlin},
  year      = {2006},
  pages     = {79--98},
  doi       = {10.1007/3-540-33484-X_4},
  eprint    = {gr-qc/0410006},
  archivePrefix = {arXiv}
}

@article{Flores_Sanchez_2008,
  author  = {Flores, Jos{\'e} L. and S{\'a}nchez, Miguel},
  title   = {The causal boundary of wave-type spacetimes},
  journal = {Journal of High Energy Physics},
  volume  = {2008},
  number  = {03},
  pages   = {036},
  year    = {2008},
  eprint  = {0712.0592},
  archivePrefix = {arXiv}
}

@article{Hubeny_Rangamani_2002,
  author  = {Hubeny, Veronika E. and Rangamani, Mukund},
  title   = {Causal structures of pp-waves},
  journal = {Journal of High Energy Physics},
  volume  = {2002},
  number  = {12},
  pages   = {043},
  year    = {2002},
  eprint  = {hep-th/0211195},
  archivePrefix = {arXiv}
}

@article{Hubeny_Rangamani_Ross_2004,
  author  = {Hubeny, Veronika E. and Rangamani, Mukund and Ross, Simon F.},
  title   = {Causal inheritance in plane wave quotients},
  journal = {Physical Review D},
  volume  = {69},
  pages   = {024007},
  year    = {2004},
  doi     = {10.1103/PhysRevD.69.024007},
  eprint  = {hep-th/0307257},
  archivePrefix = {arXiv}
}

@article{Hubeny_Rangamani_Ross_2005,
  author  = {Hubeny, Veronika E. and Rangamani, Mukund and Ross, Simon F.},
  title   = {Causally pathological spacetimes are physically relevant},
  journal = {International Journal of Modern Physics D},
  volume  = {14},
  number  = {12},
  pages   = {2227--2232},
  year    = {2005},
  doi     = {10.1142/S0218271805007760},
  eprint  = {gr-qc/0504013},
  archivePrefix = {arXiv}
}

@article{GarciaParrado_Senovilla_2005,
  author  = {Garc{\'\i}a-Parrado, Alfonso and Senovilla, Jos{\'e} M. M.},
  title   = {Causal structures and causal boundaries},
  journal = {Classical and Quantum Gravity},
  volume  = {22},
  number  = {9},
  pages   = {R1--R84},
  year    = {2005},
  doi     = {10.1088/0264-9381/22/9/R01},
  eprint  = {gr-qc/0501069},
  archivePrefix = {arXiv}
}

@misc{Minguzzi_2012,
  author        = {Minguzzi, Ettore},
  title         = {Causality of spacetimes admitting a parallel null vector and weak {KAM} theory},
  year          = {2012},
  eprint        = {1211.2685},
  archivePrefix = {arXiv},
  primaryClass  = {gr-qc}
}
\end{document}